%% file: ShadowLineDistribution.tex
\documentclass{amsart}[12pt]
\usepackage[usenames]{xcolor}
\usepackage{amssymb}
\usepackage{amsthm}
\usepackage{diagrams}
\usepackage{arydshln}
\usepackage{mathrsfs}
\usepackage{graphicx}
\usepackage{amsmath}
\usepackage{epstopdf}
\usepackage{url}
\usepackage[all]{xy}
\usepackage{todonotes}
\usepackage{caption}
\usepackage{subcaption}
\definecolor{brightcerulean}{rgb}{.26, .41, .88}

\usepackage[ allcolors=brightcerulean, allbordercolors=brightcerulean,
        colorlinks, citecolor=brightcerulean,
        pdfauthor={},
        pdftitle={}
]{hyperref} 

\include{macros}

\newcommand{\G}{\mathbf{G}}
\newcommand{\bP}{\mathbb{P}}
\newcommand{\bpsi}{{\overline\psi}}
\newcommand{\fE}{{\widehat E}}

\usepackage{pifont}

\newcommand{\ecl}[1]{\href{https://www.lmfdb.org/EllipticCurve/Q/#1/}{\texttt{#1}}}

\title{Shadow Line Distributions}

\author{Jennifer S. Balakrishnan}

\address{Jennifer S. Balakrishnan, Department of Mathematics \& Statistics, Boston University, 665 Commonwealth Avenue, Boston, MA 02215, USA}

\email{jbala@bu.edu}

\author{Mirela \c{C}iperiani}

\address{Mirela \c{C}iperiani, Department of Mathematics, The University of Texas at Austin, 1 University Station, C1200, Austin, Texas 78712, USA}

\email{mirela@math.utexas.edu}

\author{Barry Mazur}

\address{Barry Mazur, Department of Mathematics, Harvard University, 1 Oxford Street, Cambridge, MA, 02138, USA}

\email{mazur@g.harvard.edu}

\author{Karl Rubin}

\address{Karl Rubin, Department of Mathematics, UC Irvine, Irvine, CA 92697, USA}

\email{krubin@math.uci.edu}

\begin{document}

\begin{abstract}Let $E$ be an elliptic curve over $\Q$ with Mordell--Weil rank $2$ and $p$ be an odd prime of good ordinary reduction. For every imaginary quadratic field $K$ satisfying the Heegner hypothesis, there is (subject to the Shafarevich--Tate conjecture) a line, i.e., a free $\Z_p$-submodule of rank $1$,  in $ E(K)\otimes \Z_p$ given by  universal norms coming from the Mordell--Weil groups of subfields of the anticyclotomic $\Z_p$-extension of $K$; we call it the {\it shadow line}. When the twist of $E$ by $K$ has analytic rank $1$, the shadow line is conjectured to lie in $E(\Q)\otimes\Z_p$; we verify this computationally in all our examples. We study the distribution of shadow lines in $E(\Q)\otimes\Z_p$ as $K$ varies, framing conjectures based on the computations we have made.
\end{abstract}

\date{\today}
\maketitle
\tableofcontents

\section*{Perspective}

\footnotetext{
2010 Mathematics Subject Classification. 11G05, 11G50, 11Y40. \\
\indent \, \,Key words and phrases. Elliptic curve, Universal Norm, Anticyclotomic p-adic Height, Shadow Line 
\medskip
}

There were hints of this in the work of Jacobi before, but it was  Poincar{\'e} in his 1901 paper {\it Sur les propri{\'e}t{\'e}s arithm{\'e}tiques des courbes alg{\'e}briques} \cite{Poincare} who pointed out that the set of rational points on an elliptic curve has a natural (abelian)  {\it group structure}:
\begin{quote} 
{\'E}tudions d'abord la distribution des points rationnels sur ces courbes.
J'observe que la connaissance de deux points rationnels sur une cubique rationnelle suffit pour en faire conna\^ itre un troisi{\`e}me. 
 \end{quote}
 Even though Poincar{\'e} also suggests in his paper that the group of rational points on an elliptic curve is finitely generated, it took two decades before this was actually proved  to be the case{\footnote{(Mordell, 1922). Such groups of rational points,  a bit later, were called  Mordell--Weil groups, since Andr{\'e} Weil had generalized Mordell's Theorem to prove that the group of $K$-rational points of any abelian variety over any number field $K$ is finitely generated.}}, allowing us to focus on the finite fundamental invariant: the {\it rank} of the group of rational points.

 Almost half a century later, N{\'e}ron and Tate---by defining the canonical (\emph{N{\'e}ron-Tate}) height on any abelian variety over a number field---established a further canonical structure on the Mordell--Weil groups of elliptic curves:  the quotient of such Mordell--Weil groups  by their torsion subgroups can be viewed canonically, up to orthogonal isometry, as {\it discrete  lattices in Euclidean space of dimension equal to their rank}.  Subsequently, analogous (canonical)  $p$-adic height inner products were defined on Mordell--Weil groups for all prime numbers $p$.   
 
 All this provides an intricate interlacing structure on  what might have initially been regarded to be a simple arithmetic feature of an elliptic curve over $\Q$: its {\it  set of rational points.}  The object of this paper is to consider  further arithmetic architecture canonically constructible on this set.   Namely, for any elliptic curve $E$ over $\Q$ with Mordell--Weil rank $2$ and for any prime number $p$ we will be defining (canonically)  a web consisting  of (we conjecture: {\it infinitely many}) $\Z_p$-lines in $E(\Q)\otimes{\Z_p}$ coming from the Mordell--Weil behavior of $E$ over a specific set of (we conjecture: correspondingly infinitely many) quadratic imaginary fields.
 \bigskip

\section{Introduction: the shadow lines} \label{intro} Let $E/\Q$ be an elliptic curve of analytic rank $2$ and $p$ an odd prime of good ordinary reduction such that the rational $p$-torsion $E(\Q)_p$ is trivial. Assume that the $p$-primary part of the Shafarevich--Tate group of $E/\Q$ is finite. Then consider an imaginary quadratic field $K$ such that the analytic rank of $E/K$ is $3$ and the Heegner hypothesis holds for $E$, i.e., all primes dividing the conductor of $E/\Q$ split in $K$. We are interested in the subspace of $E(K) \otimes \Zp$ generated by the anticyclotomic universal norms. 

 To define this space, let $K_\infty$ be the anticyclotomic $\Zp$-extension of $K$ and let $K_n$ denote the subfield of $K_\infty$ whose Galois group over $K$ is isomorphic to $\Z/p^n\Z$.  The module of \textit{universal norms}  with respect to $K_\infty/K$ is defined by
\[
\cU_K = \bigcap_{n \geq 0} N_{K_n/K}(E(K_n) \otimes \Zp),
\]
where $N_{K_n/K}$ is the norm map induced by the map $E(K_n) \to E(K)$ given by $P \mapsto \underset{\sigma\in \mathrm{Gal}(K_n/K)}\sum P^\sigma$.
Let $L_K$ denote the $p$-divisible closure of  $\cU_K$ in $E(K) \otimes \Zp$.

By work of Cornut \cite{Cornut} (see the Theorem in the Introduction and the discussion after it) and Vatsal \cite[Theorem 1.4]{Vatsal} on the nontriviality of Heegner points we know that $\rk_{\Z_p} L_K = 1$ if the $p$-primary part of the Shafarevich--Tate group of $E/K_n$ is finite for every $n\in \bN$, see \cite[Corollary 4.4]{MRgrowth}. Complex conjugation acts on $E(K) \otimes \Zp$ and it preserves $\cU_K$. Consequently $L_K$ lies in one of the corresponding eigenspaces $E(K)^+\otimes \Zp$ and $E(K)^-\otimes \Zp$. Observe that $E(K)^+\otimes \Zp =E(\Q) \otimes \Zp$. Under our assumptions, by work of Skinner--Urban \cite{SU}, Nekov\'a\v r \cite{Nekovar}, Gross--Zagier \cite{GZ}, and Kolyvagin \cite{Kolyvagin} we know that
\[
\rk_{\Z_p} E(K)^+\otimes \Zp \geq 2 \indent \text{and}\indent \rk_{\Z_p} E(K)^-\otimes \Zp =1.
\]
Then by the Sign Conjecture \cite{MRgrowth}, we expect $L_K$ to lie in $E(\Q) \otimes \Zp$. Our main motivating questions are the following: 

\begin{Quests}[\cite{MRgrowth}]
As $K$ varies, we presumably get different shadow lines $L_K$. 
\begin{enumerate} 
\item What are these lines and how are they distributed in $E(\Q) \otimes \Zp$? 
\item Does the shadow line $L_K$ determine the quadratic field $K$ uniquely?
\end{enumerate}
\end{Quests}

In order to identify the shadow line $L_K$, we use the fact that $\cU_K$ lies in the kernel of the anticyclotomic $p$-adic height pairing $\langle \, , \, \rangle: E(K) \otimes \Zp \times E(K)  \otimes \Zp \to \Zp$, see \cite[Proposition 4.5.2]{MTpairing}. The use of this pairing forces us to assume that $p$ splits in $K/\Q$ as otherwise the pairing is trivial.  Due to the action of complex conjugation, it follows that the restriction of the pairing to either eigenspace $E(K)^\pm$ is trivial. Then since $\rk_{\Z_p} E(K)^-\otimes \Zp =1$ and if $\rk_{\Z_p} E(\Q)\otimes \Zp =2$, after verifying the non-triviality of the pairing we can deduce that $L_K$ equals the kernel of the anticyclotomic $p$-adic height pairing and $L_K \subseteq E(\Q) \otimes \Zp$, see  \cite{ShadowComp} for further details. We are hence able to identify $L_K$ by fixing a basis of $E(\Q) \otimes \Zp$ and using the anticyclotomic $p$-adic height pairing to compute the slope $s_K$ of $L_K$, see \S \ref{Comp}.

Now studying the distribution of shadow lines $L_K$ can be done by studying the variation of slopes $s_K=(x_K, y_K) \in \bP^1(\Zp)$, which we can view modulo $p^n$ for $n\in \mathbb N$. Note that equidistribution of the shadow lines $L_K$ corresponds to the statement that the values of $s_K$ modulo $p^n$ split equally among the $(p+1)p^{n-1}$ options as $n$ grows when we consider all quadratic imaginary fields of conductor up to some reasonable bound satisfying all the conditions that we have set out. 

Our computations indicate that the distribution of shadow lines in $E(\Q) \otimes \Zp$ depends on the kind of ordinary reduction of the elliptic curve at the prime $p$. (See \cite{github} for code and data.)

\begin{definition}\label{defAnomal}A prime number $p$ where $E$ has good ordinary reduction is said to be {\it anomalous} for an elliptic curve $E$ over $\Q$ if equivalently, 
\begin{itemize}
\item $E(\F_p)\otimes{\Z}_p$ is nontrivial; or 
\item  $E(\F_p)\otimes{\F}_p$ is of dimension $1$ over ${\F}_p$; or more specifically:
\item when $p \ge 5$, $E(\F_p)$ is cyclic of order $p$; \\
when $p=3$, $E(\F_p)$ is cyclic of order $3$ or $6$; \\
 when $p=2$, $E(\F_p)$ is  cyclic of order $2$ or $4$.
\end{itemize} 
Otherwise, the prime $p$ is said to be {\it non-anomalous}.
\end{definition}
 
Our computations suggest that  in the case when $p$ is a prime of good ordinary non-anomalous reduction, shadow lines are equidistributed in $E(\Q) \otimes \Zp$, see \S \ref{nonanomalous} and Conjecture \ref{uniform}.  However, this uniformity fails when $p$ is a prime of good ordinary anomalous reduction but then seems to reappear after two rounds of additional restrictions. The first of these restrictions is the consideration of the \textit{receptacle} $H \subseteq  E(\Q)\otimes \Z_p$ of shadow lines (see \eqref{defH}), as we will now describe.

Let $K''/K'$ be a cyclic Galois extension of number fields with $[K'':K']=p$  and  $\wp'$  a prime of $K'$ dividing $p$ that is ramified in $K''/K'$.  Let $\wp''$  be the prime of $K''$  lying above $\wp'$, and denote by $k_{\wp''}$ and  $k_{\wp'}$ the corresponding residue fields. Note that $k_{\wp''}=k_{\wp'}$ since $K''/K'$ is totally ramified at $p$. We have the following diagram:

\begin{equation} \label{towers}
\begin{diagram}[small]
E(K'') &\rTo^=& E(\O_{K''})& \rTo&   E(k_{\wp''}) &\rTo^= &   E(k_{\wp'})&\rTo&  E(k_{\wp'})\otimes \F_p\\
  \dTo_{\mathrm{Norm}} &     &\dTo_{\mathrm{Norm}} & &\dTo_{\mathrm{Norm}} & &\dTo_{\cdot p}&&\dTo_{\cdot p}\\
E(K') &\rTo^=& E(\O_{K'})& \rTo&   E(k_{\wp'})  &\rTo^= &   E(k_{\wp'})&\rTo&  E(k_{\wp'})\otimes \F_p\\
\end{diagram}
\end{equation}

 Since the anticyclotomic extension $K_\infty/K$ has the property that $K_\infty/K_n$ is totally ramified at all primes above $p$ for some sufficiently large $n$, it follows that some subquotient `storey'  $K_{m+1}/K_m$ of that tower is of the form of \eqref{towers}.  Consequently, the diagram
 \[
 \begin{diagram}[small]
 \cU_K &\rTo & \{0\}\\
  \dInto &     &\dInto\\  
   E(\Q)\otimes \Z_p & \rTo &  E(\F_p)\otimes \F_p\\
 \end{diagram}
 \]
\noindent is commutative. Noting that---by the known upper bounds on the number of rational points of an elliptic curve over a finite field---if $p >2$ is a prime of good reduction for $E$, we have the isomorphism 
$$E(\F_p)\otimes \Z_p \stackrel{\simeq}{\longrightarrow} E(\F_p)\otimes \F_p,$$
 we define: 
   \begin{equation}\label{defH}
  H:= \ker\big\{E(\Q)\otimes \Z_p \to E(\F_p)\otimes \Z_p\big\}\  \subseteq \ E(\Q)\otimes \Z_p
  \end{equation}  

\noindent as the \textit{receptacle} for the universal norms in $E(\Q)\otimes \Z_p $. It follows that  
\begin{itemize}
\item if $p$ is non-anomalous then $H= E(\Q)\otimes \Z_p$, and 
 \item if $p\ge 3$ is anomalous for $E$ we have that $H$ is of index $1$ or $p$ in $E(\Q)\otimes \Z_p$ depending on whether the map  $E(\Z_p)\otimes \Z_p \to  E(\F_p)\otimes \Z_p$ is trivial or not.
 \end{itemize} 

Hence in the case of an anomalous prime $p$, many shadow lines $L_K$ may lie in $H$, see \S \ref{sec:anomal}. In order to address this obstruction to the potential equidistribution of shadow lines in the case of anomalous primes, we will study the distribution of 
\[
L'_K:= L_K \cap H \text{ in } H.
\]

\medskip
We will now describe the next set of restrictions. For any imaginary quadratic field $K$ that produces a shadow line for $E$ (i.e., satisfying the Heegner hypothesis) we get the commutative diagram: 
  
\begin{equation} \label{formal}
\begin{diagram}[small]
 \cU_K  &\rInto & H& \rInto&  E(\Q)\otimes \Z_p&\rTo &  E(\F_p)\otimes \Z_p\\
  &\rdTo &  \dTo_{\psi} & &\dTo & &\dTo_{=}\\
& & \fE({\Z}_p) &\rInto &  E(\Z_p)\otimes \Z_p  & \rTo&    E(\F_p)\otimes \Z_p \\
\end{diagram}
\end{equation}  
where $\fE$ denotes the formal group of the elliptic curve $E$.
Notice that by \eqref{formal} we can derive the map 
\[
\bpsi: H \otimes \F_p \to \fE (\Z_p)\otimes \F_p.
\]
Since the algebraic rank of $E/\Q$ is $2$, when $\bpsi$ is non-trivial\footnote{Our computations indicate that this condition fails very rarely.} (see Lemma \ref{LineL}) its kernel is one-dimensional and denoted by 
\begin{equation}\label{defL}
\cL:= \ker \bpsi \subseteq  H \otimes \F_p.
\end{equation} 
We call $\cL$ the  \textit {natural line modulo} $p$. 

To our surprise, we find that most shadow lines $L'_K$ coincide with $\cL$, see Conjecture \ref{shadow-modp}, and we even find verifiable conditions that appear to guarantee that $L'_K\equiv \cL \pmod p$, see Conjecture \ref{anomalModp}. We prove this result under the assumption that the universal norms $\cU_K$ are not $p$-divisible in the receptacle $H$, see Proposition \ref{univH}. Finally, our data indicates that this is the last obstruction to the equidistribution of shadow lines for elliptic curves $E$ without a rational $p$-isogeny, see Conjecture \ref{shadow}.

We conclude our paper by summarizing the evidence that leads us to believe that a shadow line $L_K$ uniquely determines its source field $K$, see \S\ref{shadowField}.

\section{Computations, data, and set-up}\label{Comp}

Fix $E/\Q$ an elliptic curve of analytic rank $2$ and $p$ an odd prime of good ordinary reduction such that the $p$-torsion of $E(\Q)$ is trivial. Consider imaginary quadratic fields $K$ such that
\begin{enumerate}
\item the Heegner hypothesis holds for $E$ and $K$; 
\item the analytic rank of the twisted curve $E^K/\Q$ is 1 and $p$ splits in $\cO_K$.
\end{enumerate}

Under these assumptions we expect that the $\Z_p$-rank of $E(\Q)\otimes \Zp$ equals $2$, and that the shadow line $L_K$ corresponding to $(E,p,K)$  lies in $E(\Q)\otimes \Z_p$. In order to study the variation of shadow lines $L_K$ as $K$ varies, we fix a basis of $E(\Q)\otimes \Z_p$  by choosing two linearly independent points
\[
 P_1, P_2 \in E(\Q)\setminus E(\Q)_{\tors},
 \]
and then study the variation of the slope $s_K$ of $L_K$ with respect to this basis.  

In order to compute the slope $s_K$, for each quadratic field $K$ we choose a non-torsion point
 \[ 
 R\in E(K)^-, 
 \]
 where $E(K)^-$ denotes the $-$eigenspace of $E(K)$ under complex conjugation, see \cite{ShadowComp}.

Then the slope $s_K$ of the shadow line corresponding to $(E,p,K)$ in $E(\Q)\otimes \Z_p$ with respect to the basis $\{P_1, P_2\}$ is
\[
s_{K} = (-\langle P_1, R \rangle, {\langle P_2, R\rangle})\in \bP^1(\Zp),
\]
 where $\langle \cdot, \cdot \rangle$ denotes the anticyclotomic $p$-adic height pairing. Observe that once $K$ is fixed, our choice of $R$ does not affect the slope $s_K$.




\section{Distributions for non-anomalous primes $p$} \label{nonanomalous}

In this section, $p$ is a prime of non-anomalous good ordinary reduction for $E$, see Definition \ref{defAnomal}.  We considered 15 pairs $(E,p)$ and for each pair we computed $s_K$, the slope of shadow line $L_K$, for around 200-300 quadratic fields $K$, produced using the first 700 Heegner discriminants for the field.

We should mention that there is a small amount of loss  in the data: typically we had to skip a small percentage of fields in the range of computation.  These corresponded to fields $K= \Q(\sqrt D)$ for which either 

\begin{enumerate}
\item finding a non-torsion point in $E^K(\Q)$ was difficult (after carrying out $2$-, $4$-, and $8$-descents in Magma), or 
\item after the relevant descents were carried out and a non-torsion point was found, the resulting point had coordinates that were too large for our computations in the following sense: One step in the computation of shadow lines is factoring a denominator ideal in the ring of integers $\cO_K$ of the quadratic field $K$. Since factorization is very difficult, this step had a time limit in place, and some points had coordinates that were so large that this factorization was not completed within the time limit.
\end{enumerate}
We do not expect this loss to bias the resulting distribution in a significant way.

Once we computed slopes, we recorded the value of the slope modulo $p$. For example, consider the distribution of slopes for (\ecl{997.c1}, 3): from computing with 256 fields, we found that
\begin{itemize}
\item 70 fields produced slope $(0, 1) \pmod 3$,
\item 64 fields produced slope $(1,1) \pmod 3$,
\item 65 fields produced slope $(2,1) \pmod 3$,
\item 57 fields produced slope $(1, 0) \pmod 3$.
\end{itemize}

We summarize this data by recording the ordered list $[70, 64, 65, 57]$. We note that we considered 256 out of 260 eligible quadratic imaginary discriminants $D$ for (\ecl{997.c1}, 3), with $D \geq -4628$ (the other 4 eligible discriminants in this range were skipped for one of the two reasons mentioned above). 
\subsection{Non-anomalous data mod $p$}

For each pair $(E, p)$, we carried out the same process, computing the distribution of slopes modulo $p$.
We summarize the data for each $(E,p)$ in the table below with slopes listed in the following order: $ (0,1), \ldots, (p-1,1), (1,0)$. \begin{center}
 \begin{tabular}{| l |  l  | c | c | c| c ||}
    \hline
 $(E,p)$  & slope distribution mod $p$ & $D \geq $ &  eligible $D$ used & \% lost   \\ \hline
(\ecl{709.a1}, 3) & [59, 50, 61, 55] & $-4376$ & 225/241 & 6.6\% \\  
(\ecl{997.c1}, 3) & [70, 64, 65, 57]  & $-4628$  & 256/260 & 1.5\%\\

(\ecl{1627.a1}, 3) & [69, 54, 64, 54]  & $-4691$  & 241/246 & 2.0\%\\

(\ecl{2677.a1}, 3) &[50, 62, 52, 61]  & $-4559$ & 225/234 & 3.8\% \\
\hline
(\ecl{709.a1}, 5) &[48, 52, 38, 44, 43, 41] & $-4376$ & 266/276 & $3.6\%$ \\
(\ecl{1531.a1}, 5) & [36, 42, 46, 44, 44, 42] & $-4344$ & 254/269 & 5.6\%\\ 
(\ecl{1621.a1}, 5) & [43, 39, 57, 47, 49, 39] & $-4811$ & 274/280& $2.1\%$\\
(\ecl{1873.a1}, 5) & [59, 43, 43, 50, 45, 37] & $-4879$ & 277/284 & $2.5\%$ \\
(\ecl{1907.a1}, 5) &  [43, 34, 39, 32, 34, 43] & $-4004$ & 225/240 & 6.3\% \\ 
(\ecl{1933.a1}, 5) &  [39, 47, 36, 48, 57, 55] & $-4804$  & 282/288 & 2.1\% \\
\hline
(\ecl{643.a1}, 7) &  [24, 31, 24, 29, 34, 34, 33, 26] & $-3827$ & 235/248 & 5.2\%\\
(\ecl{709.a1}, 7) & [24, 33, 40, 28, 29, 24, 33, 33] &  $-3863$ & 244/255 & 4.3\% \\
 (\ecl{997.c1}, 7) & [33, 27, 24, 37, 31, 22, 29, 24] & $-3811$ & 227/233& 2.6\%\\ 
 (\ecl{1613.a1}, 7) & [44, 41, 43, 23, 33, 25, 32, 41]& $-4623$ & 282/290 & 2.8\% \\
 (\ecl{1627.a1}, 7)& [34, 41, 39, 47, 26, 33, 44, 30]  &  $-4679$ & 294/298& 1.3\%\\
 \hline 
 \end{tabular}
 \end{center}

 \bigskip
 
 This data suggests that the shadow lines are equidistributed in $E(\Q)\otimes \F_p$.
 
 \subsection{Non-anomalous data mod $p^2$}\label{secondLevel}
We now look at the distributions of $s_K$ modulo $p^2$.  We display the data for the coefficient of $p$ of the first entry of $s_K$ if $s_K \not \equiv (1,0) \pmod p$ and the corresponding data for the second entry of $s_K$ if $s_K= (1,0) \pmod p$ corresponding to the findings of the above table.

\bigskip

\begin{tabular}{cc}
    \begin{minipage}{.5\linewidth}
       \begin{tabular}{| l  |  l  |}
 \hline
 $(E,p)$ & mod $p^2$ distribution \\ \hline
(\ecl{709.a1}, 3) &  59: [18, 18, 23] \\
& 50: [14, 17, 19] \\
& 61: [19, 18, 24]  \\
& 55: [14, 18, 23] \\
 \hline
 (\ecl{997.c1}, 3) & 70: [22, 19, 29]\\
&  64: [24, 20, 20] \\
& 65: [23, 22, 20] \\
& 57: [21, 15, 21]\\
\hline
(\ecl{1627.a1}, 3) & 69: [21, 24, 24]\\
& 54: [23, 20, 11] \\
& 64: [26, 21, 17] \\
& 54: [19, 14, 21] \\
\hline
(\ecl{2677.a1}, 3) & 50: [16, 15, 19]\\
& 62: [15, 24, 23]\\
& 52: [14, 20, 18]\\
& 61: [24, 18, 19] \\
\hline
 (\ecl{709.a1}, 5) & 48: [4, 10, 12, 11, 11]\\
 & 52: [11, 12, 11, 9, 9]\\
 & 38: [6, 12, 12, 4, 4]\\
 & 44: [5, 7, 6, 15, 11]\\
 & 43: [8, 8, 9, 12, 6]\\
 & 41: [8, 7, 12, 9, 5]\\
 \hline
 (\ecl{1531.a1}, 5) & 36: [9, 2, 6, 7, 12]\\
& 42: [4, 6, 8, 12, 12]\\
& 46: [10, 9, 9, 6, 12]\\
& 44: [5, 8, 8, 16, 7]\\
& 44: [9, 11, 11, 7, 6] \\
& 42: [3, 9, 9, 8, 13] \\
 \hline
 (\ecl{1621.a1}, 5) & 43: [8, 12, 13, 5, 5]\\
 & 39: [7, 6, 10, 11, 5]\\
 & 57: [15, 12, 11, 8, 11]\\
 & 47: [8, 9, 12, 10, 8]\\
 & 49: [9, 10, 8, 13, 9]\\
 & 39:  [9, 9, 8, 4, 9] \\
 \hline
  (\ecl{1873.a1}, 5) & 59: [16, 13, 8, 12, 10]\\
  &43: [11, 10, 7, 11, 4]\\
  &43: [11, 6, 7, 10, 9]\\
  &50: [9, 7, 17, 8, 9]\\
  &45: [10, 6, 8, 12, 9]\\
  &37: [7, 8, 7, 7, 8] \\
 \hline
 (\ecl{1907.a1}, 5) & 43: [8, 7, 10, 12, 6]\\
& 34:  [3, 1, 12, 9, 9]\\
& 39:  [3, 8, 10, 11, 7]\\
& 32:  [9, 9, 5, 2, 7]\\
& 34:  [4, 4, 6, 12, 8]\\
& 43:  [7, 13, 8, 6, 9]\\
 \hline
  \end{tabular}
    \end{minipage} &

    \begin{minipage}{.5\linewidth}
     \begin{tabular}{| l |  l  |}
 \hline
 $(E,p)$ & mod $p^2$ distribution \\ \hline

 (\ecl{1933.a1}, 5) & 39: [7, 7, 7, 9, 9]\\
& 47: [7, 12, 8, 9, 11]\\
& 36: [5, 12, 5, 7, 7]\\
& 48: [9, 11, 8, 12, 8]\\
& 57: [10, 14, 10, 9, 14]\\
& 55: [15, 7, 10, 16, 7]\\
 \hline
 (\ecl{643.a1}, 7) &24: [2, 4, 4, 3, 3, 2, 6]\\
 &31: [2, 7, 8, 5, 3, 3, 3]\\
& 24: [3, 2, 1, 2, 3, 6, 7]\\
 &29: [6, 2, 6, 8, 2, 3, 2]\\
 &34: [4, 6, 7, 5, 2, 7, 3]\\ 
 &34: [7, 4, 8, 6, 4, 1, 4]\\
 &33: [2, 6, 9, 2, 5, 4, 5]\\
 &26: [1, 5, 4, 7, 5, 1, 3] \\
 \hline
 (\ecl{709.a1}, 7) & 24: [3, 1, 7, 3, 4, 0, 6] \\
 &33: [6, 6, 6, 3, 4, 5, 3]\\
 &40: [5, 5, 8, 3, 9, 5, 5]\\
 &28: [5, 6, 3, 3, 2, 6, 3]\\
 &29: [6, 2, 3, 2, 6, 7, 3]\\
 &24: [2, 8, 2, 3, 3, 3, 3]\\
 &33: [4, 5, 3, 4, 8, 5, 4]\\
 &33: [3, 4, 3, 7, 4, 7, 5]\\
 \hline
 (\ecl{997.c1}, 7) & 33: [4, 3, 7, 8, 4, 3, 4]\\
 & 27: [4, 3, 4, 3, 4, 7, 2]\\
 & 24: [4, 2, 3, 4, 1, 5, 5]\\
 & 37: [8, 3, 6, 5, 5, 3, 7]\\
 & 31: [5, 7, 5, 2, 4, 4, 4]\\
 & 22: [5, 3, 5, 3, 1, 2, 3]\\
 & 29:  [3, 9, 4, 3, 3, 2, 5]\\
 & 24: [4, 5, 3, 1, 5, 3, 3]\\
 \hline
 (\ecl{1613.a1}, 7) & 44: [2, 5, 8, 7, 5, 10, 7]\\
& 41: [7, 6, 4, 8, 3, 7, 6] \\
& 43: [8, 2, 5, 10, 6, 5, 7]\\
& 23: [2, 2, 2, 2, 3, 6, 6]\\
& 33: [4, 5, 4, 8, 6, 3, 3]\\
& 25: [0, 2, 5, 4, 5, 4, 5]\\
& 32: [10, 5, 4, 3, 1, 4, 5]\\
& 41: [6, 2, 4, 6, 10, 6, 7]\\
 \hline
 (\ecl{1627.a1}, 7) & 34: [3, 8, 6, 6, 2, 4, 5] \\
 & 41: [8, 3, 8, 4, 6, 4, 8]\\
 & 39:  [7, 6, 7, 2, 5, 8, 4]\\
 & 47:  [4, 6, 12, 5, 8, 4, 8]\\
 & 26:  [4, 6, 2, 4, 3, 4, 3]\\
 & 33:  [4, 6, 4, 5, 4, 5, 5]\\
 & 44:  [6, 5, 4, 8, 8, 7, 6]\\
 & 30:  [8, 2, 4, 4, 5, 2, 5] \\
 \hline
 \end{tabular}

    \end{minipage} 
\end{tabular}
 
\begin{conjecture}\label{uniform}
Let $E/\Q$ be an elliptic curve of analytic rank $2$, $p$ an odd prime of good ordinary non-anomalous reduction, and $K$ an imaginary quadratic field satisfying the Heegner hypothesis for $E$ such that the analytic rank of the twisted curve $E^K/\Q$ is 1 and $p$ splits in $K$. 

Then the distribution of shadow lines $L_K$ in $E(\Q)\otimes \Z_p$ is uniform.
\end{conjecture}


\section{Distributions for anomalous primes $p$} \label{sec:anomal}

In this section, we consider primes $p$ where $E$ has good ordinary \emph{anomalous} reduction, see Definition \ref{defAnomal}. Repeating the same process as in \S\ref{nonanomalous} for anomalous primes $p$ produces visibly different distributions.  We started the investigation in the anomalous case considering the position of shadow lines in $E(\Q)\otimes \Z_p$ but soon realized that this had to be refined. Here is one example illustrating what we observed:

\begin{example}
Consider (\ecl{709.a1}, 29): we compute the slopes of shadow lines for $208$ imaginary quadratic fields $K$ satisfying the necessary hypotheses (with $D \geq  -3012$; we skip $10$ such values of $D$ in the eligible set). Here we find that all $208$ slopes $s_K=(x_K, 1)$  with $x_K= 13 +29a_K$ where $a_K\in \Zp$. 

When we view the slopes modulo $p^2$ we again observe a bias because we find $200$ fields $K$ with $a_K\equiv 24 \pmod{29}$, $2$ fields with $a_K \equiv 27 \pmod{29}$, and only $1$ field for each of $a_K \in \{ 2,7,8,10, 11,13\}$ modulo $29$.
  
\end{example}

The modulo $p$ bias is easy to explain. As described in \S\ref{intro}, since the anticyclotomic $\Z_p$-extension is eventually totally ramified, the module of universal norms reduces to $0$ in $E(\F_p)\otimes \Z_p$. Hence, universal norms lie in the receptacle $H=\ker\big\{E(\Q)\otimes \Z_p \to E(\F_p)\otimes \Z_p\big\}\  \subseteq \ E(\Q)\otimes \Z_p$
which, in the case of anomalous primes $p$, may differ from $E(\Q)\otimes \Z_p$.  

We now record how we compute the receptacle $H$.

\begin{remark}\label{H-comp}
Let $E/\Q$ be an elliptic curve of algebraic rank $2$, and $p$ an odd prime of anomalous good ordinary reduction for $E$.   Consider $P_1, P_2 \in E(\Q)$ such that 
\[
E(\Q) =\Z P_1+ \Z P_2 + E(\Q)_{\tors}.
\]
Set $a_i$ to be the $p$-part of the order of the image of $P_i$ in $E(\F_p)$. Depending on the values of $(a_1, a_2)$, $H$ is generated\footnote{Notice that since $\# E(\F_p)\leq 2p <p^2$ it follows that the $p$-primary torsion of $E(\F_p)$ has order $p$ and $a_i \in \{1,p\}$.} by $\{Q_1, Q_2\}$ defined as follows:

\begin{itemize}
\item{\bf{ Case 1}}: \indent $H=E(\Q)\otimes \Z_p$
\[
a_1 = a_2 = 1: \{Q_1, Q_2\}:=\{P_1, P_2\}.
\]
\item {\bf{Case 2}}:\indent $H \subsetneq E(\Q)\otimes \Z_p$

\begin{flushleft}
 (a) \[\begin{cases} a_1 = 1, a_2 = p:\{Q_1, Q_2\}:= \{P_1, pP_2 \}, \\ a_1 = p, a_2 = 1: \{Q_1, Q_2\}:=\{ p P_1, P_2 \}, \end{cases}\]\\
 (b) \[a_1 = a_2 = p: \{Q_1, Q_2\}:=\{ P_1 + cP_2, p P_2 \}\] 
 where $c \in \Z$ such that the order of the image of $P_1+ c P_2$ in $E(\F_p)$ is coprime to $p$.\\
\end{flushleft}

\end{itemize}
\end{remark}

\begin{remark} Here are some statistics about how the pairs $(E,p)$ of elliptic curves $E$ and relevant primes $p$ sort themselves into the above three cases. Consider all rank 2 elliptic curves over $\Q$ with conductor less than 500,000 and, for each curve, all odd primes of good ordinary anomalous reduction less than 100. There are $304515$  such pairs, and here is the breakdown following the cases of the previous remark: 
\begin{itemize} 
\item Case 1: $1857/304515 \approx 0.6\%$ of all pairs;
\item Case 2: $302658/304515 \approx 99.4\%$ of all pairs, within which, due to our choice of points $P_1, P_2$, we have 
\begin{itemize} 
\item Case 2(a): $54449/304515\approx 17.9\%$ of all pairs,
\item Case 2(b): $248209/304515 \approx 81.5\%$ of all pairs.
\end{itemize}
\end{itemize}\end{remark}

\begin{remark}
Note that Case 1 produces no change in the slope distributions, since in this case, $H = E(\Q)\otimes \Z_p$. However, in Case 2(a), for example, when $a_1 = 1, a_2 = p$, the slopes we produce considering the shadow line $L'_K= L_K \cap H$ in the receptacle $H$ versus the shadow line $L_K$ in $E(\Q)\otimes \Z_p$ have an extra factor $p$ in the second component unless $L_K=\Z_pP_2$.
\medskip

Observe that in the most common Case 2(b),  the relation between the slope $s_K=(x_K, y_K)\in \bP^1(\Z_p)$ of the shadow line $L_K$ viewed in $E(\Q)\otimes \Z_p$ is related to the slope  $s'_K$ of the shadow line $L'_K$ viewed in $H$ as follows: 
\[ 
s'_K = ({x_K}-c{y_K}, py_K).
\]
Then since the constant $c$ depends on the pair $(E, p)$ but not on the quadratic field $K$, the bias modulo $p^2$ that we saw in the above data for (\ecl{1483.a1}, 31) survives modulo $p$.
\end{remark}

We will now display some data about the distribution of the slopes of shadow lines $L'_K= L_K \cap H$ in the receptacle $H$. For clarity, we review our setup and set some notation.  Let $Q_1, Q_2$ be generators of $H$ computed as described in Remark \ref{H-comp}, and let $R$ be a non-torsion point of $E(K)^-$. Then we compute slopes $s'_K$ of shadow lines $L'_K$ in $H$:
\[
s'_K=(-{\langle Q_1, R\rangle},{\langle Q_2, R\rangle})
\]
for each eligible imaginary quadratic field $K$. We record the distribution of slopes of shadow lines modulo $p$ below.

\subsection{Anomalous data mod $p$} \label{anomalMODp}

\begin{center}
 \begin{tabular}{| l  | c | l  | c | c | c |c  |}
    \hline
 $(E,p)$ & case & slope distribution mod $p$  &  mode & $D \geq$ & $D$ used & $\%$ lost  \\ 
 \hline     

 (\ecl{433.a1}, 3)   & case 2  &[25, 21, 26, 208]& $(1,0)$ & $-5240$ & 280/299& 6.4\% \\
 (\ecl{643.a1}, 3)     & case 2  & [25, 28, 139, 36]  &  (2,1) & $-4520 $ & 228/239 &  4.6\%  \\ 
 (\ecl{1058.a1}, 3) & case 2& [23, 25, 20, 25]& $?$ & $-8015$ & 93/150& 38\% \\ 
 (\ecl{1483.a1}, 3) & case 2 & [32, 147, 28, 29] & (1,1) & $-4631$ & 236/247& 4.5\% \\
 (\ecl{1613.a1}, 3) & case 2 & [24, 164, 31, 50] & (1,1) & $-4631$ &269/276 & 2.5\%\\
 (\ecl{1933.a1}, 3) & case 2 & [43, 24, 170, 33] & (2,1)& $-4835$ & 270/272& 0.7\% \\
 (\ecl{6293.d1}, 3) & case 2 & [23, 21, 22, 46] & $(1,0)$ & -12899 & 112/149 & 24.8\% \\
 (\ecl{36781.b1}, 3)  & case 1   &  [33, 24, 116, 19] & (2,1) & $-3923$ & 192/206 & 6.8 \% \\ 
\hline

 (\ecl{433.a1}, 5)     & case 2   &  [21, 8, 13, 193, 11, 16] & (3,1) & $-4631$ & 262/272 & 3.7\%  \\
 (\ecl{563.a1}, 5)     & case 2   & [14, 17, 170, 16, 10, 8] & (2,1) & $-3199$ & 235/261 & 10.0\%  \\ 
 (\ecl{997.c1}, 5) &case 2& [10, 17, 23, 15, 192, 14] &(4,1) & $-4619$ & 271/273& 0.7\%\\
  \hline
  
  (\ecl{6011.a1}, 7) & case 2 & [13, 9, 11, 7, 226, 8, 10, 5] & (4,1) & $-4591$ & 289/298 & 3.0\% \\
 \hline
 (\ecl{2251.a1}, 11) & case 2 &[2, 1, 3, 3, 2, 2, 2, 3, 181, 2, 4, 0]& (8,1) & $-3559$ & 205/235  & 12.8\% \\ 
\hline
 (\ecl{1933.a1}, 13) & case 2 & [2, 4, 2, 2, 1, 1, 1, 4, 3, 2, 1, 8, 229, 4] & (12,1) & $-4835$ &264/275 & 4.0\%  \\
\hline
 
 (\ecl{709.a1}, 29)  &  case 2    & \tiny{$[0, 0, 1, 0, 0, 0, 0, 1, 1, 0, 1, 1, 0, 1, 0, 0, 0, 0, 0, 0, 0, 0, 0, 0, 196, 0, 0, 2, 0, 0]$} & (24,1)  & $-3012$& 204/218& 6.4\% \\

 \hline
(\ecl{1483.a1}, 31)  & case 2   &  \tiny{$[1, 1, 0, 0, 0, 1, 0, 0, 0, 0, 0, 0, 0, 0, 0, 0, 1, 0, 0, 0, 196, 0, 0, 1, 0, 0, 2, 0, 0, 0, 0, 0]$} & (20,1) & $-3080$ &203/224 & 9.4\%  \\ 
\hline
\end{tabular}
\end{center}
\begin{remark}Most of these distributions were computed starting from a list of 700 Heegner discriminants (with the exception of (\ecl{433.a1}, 3), which used 800 Heegner discriminants and (\ecl{36781.b1}, 3), which used 600 Heegner discriminants). Most were run with a timeout of 600 seconds for the factorization of the denominator ideal, with the exception of (\ecl{36781.b1}, 3) and (\ecl{1058.a1}, 3) which were run with a timeout of 1800 seconds.\end{remark}

These distributions, with the exception of the one for (\ecl{1058.a1}, 3), look non-uniform, with a particular mod $p$ value hit more often than others. We will return to a discussion of the distribution of slopes for (\ecl{1058.a1}, 3) as well as for (\ecl{6293.d1}, 3) in a moment.

\medskip

 Let $s$ be the mode in the modulo $p$ shadow line distribution, i.e., for the majority of fields $K$ the image of the shadow line $L'_K$ in  $H\otimes \F_p$ coincides with the following line in $ H\otimes \F_p$:
\[
\cS:= 
\begin{cases}
(Q_1 + s_0Q_2)\F_p \indent \text{ if } s=(s_0,1)\\
Q_2  \indent \text{ if } s=(1,0).
\end{cases}
\]
We refer to $\cS$ as the \emph{distinguished shadow line modulo $p$} of the elliptic curve $E$.

We will now see how the distinguished shadow line $\cS$ relates to the natural line $\cL$ in $H\otimes \F_p$ defined in \eqref{defL} without reference to quadratic extensions $K$. The natural line $\cL$ is defined only when the map $\bpsi: H \otimes \F_p \to \fE (\Z_p)\otimes \F_p$ is non-trivial. Hence we must now understand the conditions under which this non-triviality holds.

\bigskip

Let $\G$ denote the image of the reduction map $E(\Q)\otimes \Z_p \to E(\F_p)\otimes \Z_p$. We have 
\begin{equation} \label{map-psi}
\begin{diagram}
0 &\rTo&H & \rTo&  E(\Q)\otimes \Z_p &\rTo &  \G &\rTo& 0\\
   &     & \dTo_\psi &&\dTo_{\res_p} &&\dTo_\id&&\\
0 &\rTo& \fE(\Z_p)\otimes \Z_p& \rTo&  E(\Q_p)\otimes \Z_p &\rTo &  E(\F_p)\otimes \Z_p & \rTo & 0\\
\end{diagram}
\end{equation}
where $\psi$ is the natural map that makes the first square commute. We assume throughout that $E(\Q)_p=0$. Then since $E(\F_p)\otimes \Z_p\simeq \Z/p\Z$ , it follows that 
\[
\G=\begin{cases} 0 \qquad\qquad\qquad\qquad\quad\;\;\; \text{in  Case 1}, \\
E(\F_p)\otimes \Z_p\simeq \Z/p\Z_p  \indent \text{in Case 2}.
\end{cases}
\]
Now by considering multiplication-by-$p$ maps on both horizontal exact sequences of \eqref{map-psi}, we get

\begin{equation} \label{map-bpsi}
\begin{diagram}[small] 
   &      &   0           &\rTo   & \G                     &\rTo&H\otimes\F_p                   & \rTo  & E(\Q)\otimes \F_p    &\rTo &  \G &\rTo& 0\\
   &      &                  &        & \dTo_\id               &       &\dTo_\bpsi                         &          &\dTo_{\res_p}           &        &\dTo_\id&&\\
0 &\rTo&   E(\Q_p)_p&\rTo &  E(\F_p)_p & \rTo &\fE(\Z_p)\otimes \F_p & \rTo&  E(\Q_p)\otimes \F_p &\rTo &  E(\F_p)\otimes \F_p & \rTo & 0.\\
\end{diagram}
\end{equation}

\begin{lemma} \label{LineL}
The map $\bpsi: H \otimes \F_p \to \fE (\Z_p)\otimes \F_p$ is non-trivial if and only if the following hold:
\begin{enumerate}
\item $E(\Q_p)_p=0$ and in Case 1, $\res_p E(\Q) \not\subseteq~ p^2E(\Q_p)$.
\item $E(\Q_p)_p\simeq \Z/p\Z$ and
\begin{enumerate}
\item in  Case 1, $\res_p E(\Q) \not\subseteq~ pE(\Q_p)$.
\item  in Case 2, $ res_p (P) \not\in~ p E(\Q_p)$, where $P\in E(\Q)$ is the point defined in \eqref{P-def} and $res_p: E(Q) \to E(Q_p)$.
\end{enumerate}
\end{enumerate}

\end{lemma}

\begin{proof}
Let us start by considering Case 1. We have two possibilities to analyze:
\begin{itemize}
\item $E(\Q_p)_p=0$: \indent Then $E(\Q_p)\otimes \Z_ p\simeq \Z_p$  and $\fE(\Z_p)\otimes \Z_p = p E(\Q_p)\otimes \Z_ p$ under the natural map  $\fE(\Z_p) \into E(\Q_p)$.  The diagram \eqref{map-bpsi} becomes

\begin{diagram}[small]
   &         & 0                    &\rTo&H\otimes\F_p                   & \rTo  & E(\Q)\otimes \F_p    &     \rTo& 0 &        &\\
   &        & \dTo_\id               &       &\dTo_\bpsi                         &          &\dTo_{\res_p}           &        &\dTo_\id&&\\
0 &\rTo &  \Z/p\Z & \rTo &\fE(\Z_p)\otimes \F_p & \rTo&  E(\Q_p)\otimes \F_p &\rTo &  \Z/p\Z & \rTo & 0\\
\end{diagram}
\noindent and we see that in this case the image of $\bpsi$ is trivial exactly when $\res_p E(\Q) \subseteq~ p^2 E(\Q_p)$ . \\

\item $E(\Q_p)_p\simeq \Z/p\Z$:  \indent Then $E(\Q_p)\otimes \Z_ p\ = \fE(\Z_p)\otimes \Z_p \oplus E(\Q_p)_p$. The diagram \eqref{map-bpsi} becomes

\begin{diagram}[small]
 0 &\rTo&H\otimes\F_p                   & \rTo  & E(\Q)\otimes \F_p    &     \rTo& 0 &        &\\
    &       &\dTo_\bpsi                         &          &\dTo_{\res_p}           &        &\dTo_\id&&\\ 
0  & \rTo &\fE(\Z_p)\otimes \F_p & \rTo&  E(\Q_p)\otimes \F_p &\rTo &  \Z/p\Z & \rTo & 0\\
\end{diagram}
\noindent and we see that in this case the image of $\bpsi$ is trivial exactly when $\res_p E(\Q) \subseteq~ pE(\Q_p)$.   \\
\end{itemize}

We now consider Case 2. 

\begin{itemize}
\item $E(\Q_p)_p=0$: Then the diagram \eqref{map-bpsi} gives the following

\begin{diagram}[small]
0 &\rTo  &\Z/p\Z                   &\rTo&H\otimes\F_p    \\
   &        & \dTo_\id               &       &\dTo_\bpsi       \\
0 &\rTo &  \Z/p\Z & \rTo &\fE(\Z_p)\otimes \F_p \\
\end{diagram}
which implies that the map $\bpsi$ is always non-trivial in these cases.\\

\item $E(\Q_p)_p\simeq \Z/p\Z$:  Then the diagram \eqref{map-bpsi} becomes
\begin{diagram}[small] 
   &      &   0           &\rTo   & \Z/p\Z                  &\rTo&H\otimes\F_p                   & \rTo  & E(\Q)\otimes \F_p    &\rTo &  \Z/p\Z &\rTo& 0\\
   &      &                  &        & \dTo_\id               &       &\dTo_\bpsi                         &          &\dTo_{\res_p}           &        &\dTo_\id&&\\
0 &\rTo&  \Z/p\Z&\rTo &  \Z/p\Z & \rTo &\fE(\Z_p)\otimes \F_p & \rTo&  E(\Q_p)\otimes \F_p &\rTo & \Z/p\Z & \rTo & 0\\
\end{diagram}
\noindent Set 
\begin{equation}\label {P-def}
P:= \begin{cases}
P_i & \text { in Case 2(a) and the order of  } P_i \text{ in }  E(\F_p) \text{ is coprime to } p,\\
P_1 + cP_2  & \text { in Case 2(b)}.
\end{cases}
\end{equation}
and observe that the image of $\bpsi$ is isomorphic to the subgroup of $ E(\Q_p)\otimes \F_p$ generated by the image of $\res_p (P)$. Hence $\bpsi$ is trivial if and only if $ \res_p (P) \in~ p E(\Q_p)$.

\end{itemize}

\end{proof}

We now give an explicit description of the natural line $\cL$. Consider the isomorphism
\begin{align*}
\varphi:  \fE (\Z_p)\otimes \F_p &\to p\Z_p/p^2\Z_p \\
P &\mapsto -\frac{x(P)}{y(P)}.
\end{align*}

\begin{remark}
Let $Q_1,Q_2$ be generators of the receptacle $H$ and $n_1, n_2$ be the orders of their images in $E(\F_p)$, respectively. Then the natural line $\cL$ is generated by
\[
\begin{cases}
Q_2 &  \text{if  } \varphi(\res_p (n_2Q_2)) = 0,\\
Q_1+k Q_2 & \text{else, with  } k\in \F_p \text{ such that } \varphi(\res_p(n_1Q_1)) + (k n_1/n_2) \varphi(res_p(n_2Q_2)) =0.\\
\end{cases}
\]
\end{remark}

\begin{example} We compute the natural line $\cL$ for  (\ecl{433.a1}, 5). The generators of the receptacle $H$ are 
\begin{align*}
Q_1 &= \left(-\frac{21}{25}, \frac{148}{125}\right)\\
Q_2 &=  \left(-\frac{32832}{66049}, -\frac{12229343}{16974593}\right).\end{align*}
Their additive orders in $E(\F_5)$ are 1 and 2, respectively.  We find that $$Q_1 + 4\cdot 2Q_2 = Q_1 + 3Q_2$$ generates the natural line $\cL$. Moreover, the shadow line slope distribution is
\[
[21, 8, 13, 193, 11, 16] ,
\]
which has mode $(3,1)$ and thus the distinguished shadow line $\cS$ is also generated by $Q_1 + 3 Q_2$. Hence $\cS = \cL$. 
\end{example}

Here is a table for more pairs $(E,p)$, showing the relationship between these two lines. Note that the last column contains the slope of the natural line $\cL$ if the stated equality is either unknown or false.

\bigskip

\bigskip
\begin{center}
 \begin{tabular}{| l  | l | c  | c | c | c |c  |}
    \hline
 $(E,p)$  & slope distribution mod $p$  & mode &  $\mathcal{L}  = \mathcal{S}$ \\ 
 \hline     
 (\ecl{433.a1}, 3)   &  [25, 21, 26, 208] & $(1,0)$ & $\checkmark$  \\
 (\ecl{643.a1}, 3)     &  [25, 28, 139, 36] &  (2,1) &  $\checkmark$  \\ 
 (\ecl{1058.a1}, 3) & [23, 25, 20, 25] & $?$ & (1,1) \\
 (\ecl{1483.a1}, 3) & [32, 147, 28, 29]   & (1,1) & $\checkmark$  \\
 (\ecl{1613.a1}, 3) &  [24, 164, 31, 50]  & (1,1) &$\checkmark$  \\
 (\ecl{1933.a1}, 3) &  [43, 24, 170, 33] & (2,1) & $\checkmark$  \\
  (\ecl{6293.d1}, 3) & [23, 21, 22, 46] & $(1,0) $ & (0,1) no \\
 (\ecl{36781.b1}, 3)  &   [33, 24, 116, 19] & (2,1) & $\checkmark$  \\ 
\hline
(\ecl{433.a1}, 5)     &   [21, 8, 13, 193, 11, 16] & (3,1) & $\checkmark$   \\
 (\ecl{563.a1}, 5)     &  [14, 17, 170, 16, 10, 8]  & (2,1) &  $\checkmark$  \\ 
 (\ecl{997.c1}, 5) & [10, 17, 23, 15, 192, 14] & (4,1) & $\checkmark$ \\
 \hline
 (\ecl{6011.a1}, 7) &  [13, 9, 11, 7, 226, 8, 10, 5] & (4,1) & $\checkmark$  \\
 \hline
(\ecl{2251.a1}, 11) &[2, 1, 3, 3, 2, 2, 2, 3, 181, 2, 4, 0] & (8,1) & $\checkmark$ \\ 
\hline
 (\ecl{1933.a1}, 13) & [2, 4, 2, 2, 1, 1, 1, 4, 3, 2, 1, 8, 229, 4] & (12,1) &   $\checkmark$ \\
\hline
 (\ecl{709.a1}, 29)     & \tiny{$[0, 0, 1, 0, 0, 0, 0, 1, 1, 0, 1, 1, 0, 1, 0, 0, 0, 0, 0, 0, 0, 0, 0, 0, 196, 0, 0, 2, 0, 0]$} & (24,1)  &  $\checkmark$\\
 \hline
(\ecl{1483.a1}, 31)    &  \tiny{$[1, 1, 0, 0, 0, 1, 0, 0, 0, 0, 0, 0, 0, 0, 0, 0, 1, 0, 0, 0, 196, 0, 0, 1, 0, 0, 2, 0, 0, 0, 0, 0]$}  & (20,1) & $\checkmark$ \\ 
\hline
\end{tabular}
\end{center}

\bigskip

\begin{conjecture} \label{shadow-modp}
Let $E/\Q$ be an elliptic curve of analytic rank $2$, $p$ an odd prime of good ordinary anomalous reduction, and $K$ an imaginary quadratic field satisfying the Heegner hypothesis for $E$ such that the analytic rank of the twisted curve $E^K/\Q$ is 1 and $p$ splits in $K$.

 If $E$ does not have a rational $p$-isogeny then the shadow lines $L'_K$ are not equidistributed in $H\otimes \F_p$ and the distinguished shadow line $\cS$ of $E$ coincides with the natural line $\cL$ in $H\otimes \F_p$.
\end{conjecture}

\begin{remark}
In the above conjecture we need to assume that $E$ does not have a rational $p$-isogeny. 
For example, in the case of $ (\ecl{1058.a1}, 3) $, the elliptic curve \ecl{1058.a1} admits a rational $3$-isogeny, and the data does not clearly identify $\cS$ and hence it is not clear that $\cS=\cL$. The elliptic curve \ecl{6293.d1} also admits a rational 3-isogeny.
\end{remark}

\subsection{Anomalous data mod $p^2$}

We now investigate the distribution of shadow lines $L'_K$ which coincide modulo $p$ with the distinguished shadow line $\cS$.  As in \S \ref{secondLevel} we look at the distributions of $s'_K$ modulo $p^2$ for the fields $K$ such that the corresponding shadow line coincides with the distinguished mod $p$ shadow line 

\begin{center}
 \begin{tabular}{| l  | l | l  |  }
    \hline
 $(E,p)$  & mod $p^2$ distribution \\ \hline     
 (\ecl{433.a1}, 3)   &  208: [71, 60, 77]  \\
 (\ecl{643.a1}, 3)     & 139: [42, 47, 50]   \\ 
 (\ecl{1483.a1}, 3) &147:  [36, 61, 50]  \\
 (\ecl{1613.a1}, 3) & 164: [45, 62, 57] \\
 (\ecl{1933.a1}, 3) & 170: [59, 57, 54]  \\
 (\ecl{36781.b1}, 3)  &116:  [40, 34, 42]   \\ 

\hline

 (\ecl{433.a1}, 5)     & 193:    [38, 37, 36, 46, 36] \\
  (\ecl{563.a1}, 5)     & 170:  [32, 34, 34, 37, 33]    \\ 
  (\ecl{997.c1}, 5) & 192:  [50, 36, 35, 33, 38] \\  
  \hline
  (\ecl{6011.a1}, 7) &  226:  [24, 37, 28, 37, 41, 27, 32]  \\
 \hline
(\ecl{2251.a1}, 11) & 181: [19, 15, 14, 16, 17, 17, 11, 14, 22, 18, 18] \\ 
\hline
 (\ecl{1933.a1}, 13) & 229:   [12, 26, 20, 15, 23, 9, 21, 18, 18, 17, 18, 22, 10]
 \\
\hline
 (\ecl{709.a1}, 29)     &  196:  [3, 13, 9, 8, 7, 7, 5, 7, 6, 6, 4, 2, 7, 7, 5, 8, 6, 8, 7, 7, 7, 9, 7, 3, 4, 7, 7, 12, 8] \\
 \hline
(\ecl{1483.a1}, 31)    &  196: [7, 7, 5, 8, 7, 3, 4, 3, 7, 8, 10, 5, 8, 7, 11, 4, 9, 7, 3, 10, 6, 9, 3, 8, 4, 4, 3, 8, 7, 7, 4]    \\ 
\hline
\end{tabular}
\end{center}

\noindent For instance, for (\ecl{433.a1}, 5) all slopes are of the form $3 + a \cdot 5 \mod 5^2$, where $a = 0, 1, 2, 3, 4$. The distribution of [38, 37, 36, 46, 36] recorded is the count that of the 193 fields we considered, 38 produced $a = 0$, 37 had $a =1$, and so on.

\begin{conjecture} \label{shadow} 
Let $E/\Q$ be an elliptic curve of analytic rank $2$, $p$ an odd prime of good ordinary anomalous reduction, and $K$ an imaginary quadratic field satisfying the Heegner hypothesis for $E$ such that the analytic rank of the twisted curve $E^K/\Q$ is 1 and $p$ splits in $K$. 

If $E$ does not have a rational $p$-isogeny then the shadow lines $L'_K$ which coincide with $\cS$ in $H\otimes \F_p$ are equidistributed in the receptacle $H\subseteq E(\Q)\otimes \Z_p$.
\end{conjecture}
\bigskip


\section{The distinguished shadow line modulo $p$ for anomalous primes $p$}

When does the shadow line attached to $(E,K,p)$ differ from $\cL$ modulo $p$? It seems there are some subtle issues to consider here (e.g., if $E$ admits a rational $p$-isogeny, then the statistics are not clear).  We would like to find conditions that rule out the non-modal values in the distribution.

\begin{definition}
 The triple $(E,p,K)$ is said to be {\it{filtered data}} if in addition to the assumptions that 
 \begin{itemize}
\item  $E/\Q$ is an elliptic curve of analytic rank $2$, 
\item $p$ is an odd prime of good ordinary anomalous reduction, 
\item $K$ is an imaginary quadratic field satisfying the Heegner hypothesis for $E$ such that the analytic rank of the twisted curve $E^K/\Q$ is 1 and $p$ splits in $K$,
\end{itemize}
it also satisfies the following two conditions:
\begin{itemize}
\item the class number of $K$ is coprime to $p$, and
\item the point $R\in E^K(\Q)$ is not $p$-divisible in $E^K(\Q_p)=E(\Q_p)$. 
\end{itemize}
The distribution of slopes that we find when we restrict to filtered data and vary the discriminant up to some manageable bound is referred to as the {\it filtered slope distribution}.  
\end{definition}

The following table shows the effect of restricting to filtered data.
\begin{center}
 \begin{tabular}{| l  | l | l  | c | c  |}
    \hline
 $(E,p)$  & slope distribution mod $p$  \\  \cdashline{2-3}
  & filtered slope distribution mod $p$ \\ \hline
  (\ecl{433.a1}, 3)   &  [25, 21, 26, 208] \\  \cdashline{2-3}
  &  [14, 18, 19, 126] \\
  \hline
 (\ecl{643.a1}, 3)   &   [25, 28, 139, 36] \\  \cdashline{2-3}
 & [\;0,\; 0,\; 102,\; \,\;0] \\  \hline
 
 (\ecl{1058.a1}, 3) & [23, 25, 10, 25]  \\   \cdashline{2-3}
 & [\;0,\; \ 7,\; 0,\; \,\;0]  \\ \hline
 
 (\ecl{1483.a1}, 3) &[32, 147, 28, 29] \\  \cdashline{2-3}
 &  [\;0,\; 110,\; 0,\; \,\;0] \\ \hline
 
 (\ecl{1613.a1}, 3) &[24, 164, 31, 50]  \\  \cdashline{2-3}
 &  [\;0,\; 133,\; 0,\; \,\;0] \\
 \hline
 
 (\ecl{1933.a1}, 3) &[43, 24, 170, 33] \\   \cdashline{2-3}
 &    [\;0,\; 0,\; 125,\; \,\;0] \\ \hline
 
 (\ecl{6293.d1}, 3) & [23, 21, 22, 46]  \\  \cdashline{2-3}
& [12, 15, 17, 0] \\ \hline
 
 (\ecl{36781.b1}, 3)  &  [33, 24, 116, 19] \\   \cdashline{2-3}
 &    [\;0,\; \;0,\; \;84,\; \,\;0] \\ 
\hline
 (\ecl{433.a1}, 5) & [21, 8, 13, 193, 11, 16] \\   \cdashline{2-3}
 &  [\;0,\; 0,\; 0,\; 175,\; 0,\; \;0]  \\
 \hline
 (\ecl{563.a1}, 5) &  [14, 17, 170, 16, 10, 8]  \\    \cdashline{2-3}
 & [\;0,\; 0,\; 151,\; 0,\; 0,\; \;0] \\ \hline
 (\ecl{997.c1}, 5) &  [10, 17, 23, 15, 192, 14] \\   \cdashline{2-3}
 & [\;0,\; \;0,\; \;0,\; 0,\; 166,\; \;0] \\   \hline
(\ecl{6011.a1}, 7) &  [13, 9, 11, 7, 226, 8, 10, 5]\\   \cdashline{2-3}
 & [\;0,\; 0,\; 0,\; 0, 213, 0,\; 0, \;0]\\
 \hline
 (\ecl{2251.a1}, 11) &  [2, 1, 3, 3, 2, 2, 2, 3, 181, 2, 4, 0]  \\  \cdashline{2-3}
&  [0, 0, 0, 0, 0, 0, 0, 0, 179, 0, 0, 0]  \\ 
\hline

\end{tabular}
\end{center}

\begin{center}
 \begin{tabular}{| l  | l | l  | c | c  |}
    \hline
 $(E,p)$  & slope distribution mod $p$  \\  \cdashline{2-3}
  & filtered slope distribution mod $p$ \\ \hline
 (\ecl{1933.a1}, 13) &   [2, 4, 2, 2, 1, 1, 1, 4, 3, 2, 1, 8, 229, 4] \\   \cdashline{2-3}
 &  [0, 0, 0, 0, 0, 0, 0, 0, 0, 0, 0, 0, 222, 0]  \\
\hline
 (\ecl{709.a1}, 29)     & \tiny{$[0, 0, 1, 0, 0, 0, 0, 1, 1, 0, 1, 1, 0, 1, 0, 0, 0, 0, 0, 0, 0, 0, 0, 0, 196, 0, 0, 2, 0, 0]$}   \\   \cdashline{2-3}
 & \tiny{$[0, 0, 0, 0, 0, 0, 0, 0, 0, 0, 0, 0, 0, 0, 0, 0, 0, 0, 0, 0, 0, 0, 0, 0, 196, 0, 0, 0, 0, 0]$} \\
 \hline
(\ecl{1483.a1}, 31)     &  \tiny{$[1, 1, 0, 0, 0, 1, 0, 0, 0, 0, 0, 0, 0, 0, 0, 0, 1, 0, 0, 0, 196, 0, 0, 1, 0, 0, 2, 0, 0, 0, 0, 0]$}  \\  \cdashline{2-3}
&\tiny{$[0, 0, 0, 0, 0, 0, 0, 0, 0, 0, 0, 0, 0, 0, 0, 0, 0, 0, 0, 0, 195, 0, 0, 0, 0, 0, 0, 0, 0, 0, 0, 0]$}\\
\hline

\end{tabular}
\end{center}

\bigskip

\begin{remark} The elliptic curve \ecl{433.a1} has non-trivial 3-torsion over $\Q_3$. All other curves $(E,p)$ in this list have trivial $p$-torsion in $E(\Q_p)$. \end{remark}
\begin{remark} The distribution for (\ecl{1058.a1}, 3) is striking due to the number of curves eliminated by the filters. Recall that \ecl{1058.a1} has a rational 3-isogeny (to \ecl{1058.a2}) but no local 3-torsion.  For (\ecl{1058.a1}, 3), the vast majority of fields considered had the corresponding point $R$ being locally 3-divisible: for this reason, 141 out of 152 fields were eliminated when considering the filtered distribution. Moreover, four fields had 3 dividing the class number. So only seven out of 152 fields considered survived the two filters, but they all produced the same slope mod $3$. \end{remark}
\begin{remark} The pair (\ecl{6293.d1}, 3) also exhibits interesting behavior. Like \ecl{1058.a1}, the curve \ecl{6293.d1} also admits a rational 3-isogeny but does not have local 3-torsion. (On the other hand, while \ecl{1058.a1} and \ecl{1058.a2} both have trivial rational 3-torsion, \ecl{6293.d2} has rational 3-torsion.) Nevertheless, here imposing the two filters seems to eliminate the modal value.
\end{remark}

\begin{conjecture} \label{anomalModp}
Let $(E, p, K)$ be a triple consisting of 
\begin{itemize}
\item $E$: elliptic curve defined over $\Q$ of analytic rank $2$; 
\item $p$: an odd prime of anomalous good ordinary reduction for $E$;
\item $K$: imaginary quadratic field satisfying the Heegner hypothesis for $E/\Q$ such that the analytic rank of $E^K/\Q$ equals $1$ and $p$ splits in $K$.
\end{itemize} 
Suppose that
\begin{itemize} 
\item $p$-torsion of $E(\Q_p)$ is trivial\footnote{One may be able to weaken this condition on $E/\Q_p$.},
\item $E$ doesn't have a rational $p$-isogeny,
\item $p$ does not divide the class number of $K$,
\item $R\in E(K)$ which generates $E^K(\Q)$ is not $p$-divisible in $E(\Q_p)$.
\end{itemize}
Then the image of the shadow line $L'_K$ in $H\otimes \F_p$ is independent of $K$ and coincides with the natural line $\cL$.
\end{conjecture}

\begin{remark} \hfill
\begin{itemize}
\item[-] Our data shows that if $E$ is an elliptic curve without a rational $p$-isogeny and trivial $E(\Q_p)_p$, the distinguished shadow line $\cS$ may still coincide with $\cL$ even if the class number of $K$ is $p$-divisible and $\res_\wp R\in pE(\Q_p)$ for $\wp|p$.  For example, for (\ecl{643.a1}, 3), and $K = \Q(\sqrt{-1691})$, the class number of $K$ is 18 and $R$ is locally 3-divisible, and yet the slope of the shadow line is $2 \pmod{3}$. 

\item[-] Observe that (\ecl{433.a1}, 3) does not fit the above conjecture because the $3$-torsion of $E(\Q_3)$ is non-trivial.  Note that in this example, for all $K$ that do not fit the conjecture, $R\in E(K)^-$ has order coprime to $3$ in $E(\F_3)$. However, Conjecture \ref{shadow} does appear to hold in this example.

\end{itemize}
\end{remark}

\begin{remark} Note that the $p$-part of the Shafarevich--Tate group of $E/\Q$ is trivial in all the examples that we have considered. Here is one example where that is not the case.  Consider the elliptic curve \ecl{55297189.a1}, which is good, ordinary, and anomalous at 3. Considering the first 53 eligible $K$ and discarding one to an incomplete 8-descent in Magma, we are left with 52 quadratic imaginary fields where we can compute the shadow line. For (\ecl{55297189.a1}, 3), the computation produces the distribution [2, 9, 5, 30] for the 46 fields it completed the computation within the allotted time.   The natural line and the distinguished shadow line coincide in this example.   The filtered distribution is [1, 8, 4, 20] and bears some resemblance to the phenomenon observed for (\ecl{433.a1}, 3).  Indeed, while this curve does not have global torsion, it has a 3-torsion point over $\Q_3$, as was the case for \ecl{433.a1}. 

\end{remark}

\begin{remark} We compiled statistics on how often the $p$-torsion of $E(\Q_p)$ is nontrivial among the set of elliptic curves that have good, ordinary, and anomalous reduction at $p$, from the Cremona database of elliptic curves in LMFDB (conductor $\leq 500000$). 

Fix a prime $p \geq 3$. We first count the number of elliptic curves of rank 2 that have good, ordinary, anomalous reduction at $p$. From this set, we count the number of curves which have at least one $p$-torsion point defined over $\Q_p$. \\
For $p = 3$, the proportion of these curves with local $p$-torsion is approximately 0.3370.\\
For $p = 5$, the proportion is approximately 0.20070\\
For $p = 7$, the proportion is approximately 0.1441.\\
For $p = 11$, the proportion is approximately 0.0934.

If we allow all elliptic curves of good, ordinary, anomalous reduction at $p$ but remove the restriction on rank, the proportions are as follows:\\
For $p = 3$, the proportion of elliptic curves that have  local $p$-torsion among the set that have good, ordinary anomalous reduction at $p$ is approximately 0.3581.\\
For $p = 5$, the proportion is approximately 0.2027.\\
For $p = 7$, the proportion is approximately 0.1439.\\
For $p = 11$, the proportion is approximately 0.0896.
\end{remark}

In our efforts to understand the data that led us to Conjecture \ref{anomalModp} we proved the following result:
\begin{proposition}\label{univH}
Let $(E,p, K)$ be a triple consisting of 
\begin{itemize}
\item $E$: elliptic curve defined over $\Q$ of analytic rank $2$; 
\item $p$: an odd prime of anomalous good ordinary reduction for $E$;
\item $K$: imaginary quadratic field satisfying the Heegner hypothesis for $E/\Q$ such that the analytic rank of $E^K/\Q$ equals $1$ and $p$ splits in $K$.
\end{itemize} 
Suppose that
\begin{enumerate} 
\item the $p$-torsion of $E(\Q_p)$ is trivial, \label{local_torsion}
\item if $[\im \left(E(\Q)\to E(\bF_p)\right)]_p =0$ then $\res_p E(\Q) \not\subseteq p^2 E(\Q_p)$.\label{reduction}
\end{enumerate}
Then if the module of universal norms $\cU_K$ is not $p$-divisible in $H$, the image of the shadow line $L'_K$ in $H\otimes \F_p$ is independent of $K$ and it coincides with the natural line $\cL$.
\end{proposition}

\begin{proof}
Observe that by Lemma \ref{LineL} assumptions \eqref{local_torsion} and \eqref{reduction} imply that $\cL= \ker\big( H \otimes \F_p \to \fE (\Z_p)\otimes \F_p\big)$ is a line.

We know that there exists $m$ such that $K_\infty/K_m$ is totally ramified at every prime above $p$. 
We fix a prime $\wp$ of $K$ above $p$ and $\wp_m$ a prime of $K_m$ above $\wp$. Then for every $n>m$ we have a unique prime $\wp_n$ of $K_n$ such that $\wp_n$ divides $\wp_{n-1}$.  Let $\cO_{\wp_n}$ denote the ring of integers of $\K_{\wp_n}$, the localization of $K_n$ at $\wp_n$. Observe that for $n>m$ we have the following:

\begin{diagram}[small]
0  & \rTo &\fE(\cO_{\wp_n}) & \rTo &  E(K_{\wp_n})\otimes \Z_p &\rTo^\pi & E(\bF_{p^k}) \otimes\Z_p & \rTo & 0\\
    &       &\dTo_{N_{K_n/K_{n-1}}}                &          &\dTo_{N_{K_n/K_{n-1}} }          &        &\dTo_{ \cdot p}&  &\\
 0 &\rTo&\fE(\cO_{\wp_{n-1}})  & \rTo  &E(K_{\wp_{n-1}})\otimes \Z_p    &     \rTo^\pi& E(\bF_{p^k}) \otimes\Z_p&        \rTo & 0\\
\end{diagram}
where $\bF_{p^k}$ is the residue field of $K_{\wp_m}$.

Consider the following two $\Z_p$-submodules of $\fE(\cO_\wp)$ and $E(K_{\wp})\otimes \Z_p$:
\[
\hat \cU_{\wp} = \bigcap_{n \geq 0} N_{K_{\wp_n}/K_\wp}(\fE(\cO_{\wp_n}))\indent \text{and} \indent \cU_{\wp} = \bigcap_{n \geq 0} N_{K_{\wp_n}/K_\wp}(E(K_{\wp_n}) \otimes \Zp).
\]
 It is clear that $ \hat \cU_{\wp} \subseteq \cU_{\wp}$ and now we will see that the other inclusion holds also. 
 
Fix $c\in \mathbb N$ such that $ E(\bF_{p^k}) \otimes\Z_p \overset{\cdot p^c}{\to}  E(\bF_{p^k}) \otimes\Z_p $ is the zero map. Consider $y\in \cU_\wp$, then $y= N_{K_{\wp_n}/K_\wp} (y_n)$ where $y_n \in E(K_{\wp_n})\otimes \Z_p$  for every $n\geq m+c$. Then we have
\[
y= N_{K_{\wp_n}/K_\wp} (y_n)=  N_{K_{\wp_{n-c}}/K_\wp} \left(N_{K_{\wp_n}/K_{\wp_{n-c}}}(y_n)\right)
\]
Using the above commutative diagram we see that $N_{K_{\wp_n}/K_{\wp_{n-c}}}(y_n)$ lies in the image of $\fE(\cO_{\wp_{n-1}})$. Hence $\hat  \cU_{\wp}= \cU_{\wp}$.

By \cite[Corollaries 4.33 \& 4.37]{Mazur} we know that 
\[ 
\fE(\cO_\wp)/ \hat \cU_\wp\simeq \Z_p/(1-u)\Z_p
\] 
where $u$ is the unit $p$-Frobenius eigenvalue of $E$. We need to determine $\ord_p (1-u)$.

Since $p$ is anomalous and $a_p=u+p/u = 1+p - \# E(\F_p)$ it follows that $\ord_p (1-u)\geq 1$. If $\ord_p (1-u)\geq 2$ then $1/u \equiv 1 \pmod {p^2}$ and 
\[
|a_p|=|u+p/u| = |1 +p + p^2 m| \leq 2\sqrt p
\]
where $m\in \Z$, which is not possible for $p\geq 3$. Hence $\fE(\cO_\wp)/ \hat \cU_\wp\simeq \Z/pZ$ which then implies that $\fE(\cO_\wp)/ \hat \cU_\wp = \fE(\cO_\wp)\otimes \F_p$.

Since $\res_\wp \cU_K \subseteq \cU_{\wp}= \hat \cU_{\wp}$ it follows that the image of $\cU_K$ in $H\otimes \F_p$ lies in the kernel of the map $\bpsi: H \otimes \F_p \to \fE (\Z_p)\otimes \F_p$. Finally since $\cU_K\not\subseteq pH$ and $\bpsi$ is non-trivial, it follows that the image of $\cU_K$ in $H\otimes \F_p$ equals $\cL$.

\end{proof}

This result leads us to the following question:

\begin{Quest}
Let $(E,p, K)$ be a triple consisting of 
\begin{itemize}
\item $E$: elliptic curve defined over $\Q$ of analytic rank $2$; 
\item $p$: an odd prime of anomalous good ordinary reduction for $E$;
\item $K$: imaginary quadratic field satisfying the Heegner hypothesis for $E/\Q$ such that the analytic rank of $E^K/\Q$ equals $1$ and $p$ splits in $K$.
\end{itemize} 
Suppose that
\begin{itemize} 
\item $p$-torsion of $E(\Q_p)$ is trivial,
\item $E$ doesn't have a rational $p$-isogeny,
\item $p$ does not divide the class number of $K$,
\item $R\in E(K)$ which generates $E^K(\Q)$ is not $p$-divisible in $E(\Q_p)$.
\end{itemize}
Does it follow that the module of universal norms $\cU_K$ not $p$-divisible in $H$?
\end{Quest}

It is not clear that the rank of the elliptic curve should play a role in this question but we include it since we have no data in other cases.

\section{The map from imaginary quadratic fields to shadow lines}\label{shadowField}

We will now address the question of whether the shadow line $L_K$ uniquely determines its source field $K$. The following table lists all the pairs of elliptic curves and primes $(E,p)$ that we have considered in this paper. For  each pair $(E,p)$ we display the number of quadratic fields for which we were able to compute the corresponding shadow lines (as well as the total number of quadratic fields that we attempted) and the minimal power $n$ of $p$ for which all the shadow lines that we computed are different from each other modulo $p^n$.

\begin{center}
 \begin{tabular}{|| l |c | c | c | c ||}
    \hline
 $(E,p)$   & $E$ anomalous at $p$? & $D \geq $ &  eligible $D$ used & min $n$: slopes are distinct mod $p^n$\\ \hline   
(\ecl{433.a1}, 3) &$\checkmark$& $-5240$ & 280/299 & 8  \\
 (\ecl{643.a1}, 3)  &$\checkmark$& $-4520 $ & 228/239 & $11$ \\
 (\ecl{709.a1}, 3) & &  $-4376$ & 225/241 &10 \\  
(\ecl{997.c1}, 3) && $-4628$  & 256/260 & 10 \\
 (\ecl{1058.a1}, 3) &$\checkmark$& $-8015$ & 93/150&  7 \\
 (\ecl{1483.a1}, 3) &$\checkmark$& $-4631$ & 236/247& 10\\
 (\ecl{1613.a1}, 3) &$\checkmark$& $-4631$ &269/276 &10 \\
 (\ecl{1627.a1}, 3) && $-4691$  & 241/246 & 10 \\
 (\ecl{1933.a1}, 3) &$\checkmark$& $-4835$ & 270/272& 10 \\
 (\ecl{2677.a1}, 3) && $-4559$ & 225/234 & $11$ \\
 (\ecl{6293.d1}, 3) &$\checkmark$&  $-12899$ & 112/149 & 7 \\
 (\ecl{36781.b1}, 3) &$\checkmark$& $-3923$ & 192/206 & 15 \\
 \hline
 (\ecl{433.a1}, 5) &$\checkmark$& $-4631$ & 262/272 & 7 \\ 
 (\ecl{563.a1}, 5) &$\checkmark$&$-3199$ & 235/261 & 8 \\
  (\ecl{709.a1}, 5) && $-4376$ & 266/276 &7 \\
 (\ecl{997.c1}, 5) &$\checkmark$& $-4619$ & 271/273&  8\\
(\ecl{1531.a1}, 5) &&$-4344$ & 254/269 &8 \\
(\ecl{1621.a1}, 5) && $-4811$ & 274/280& 7 \\
(\ecl{1873.a1}, 5) && $-4879$ & 277/284 & 7 \\
(\ecl{1907.a1}, 5) && $-4004$ & 225/240 & 8 \\ 
(\ecl{1933.a1}, 5) && $-4804$  & 282/288 &  6\\
 \hline
 (\ecl{643.a1}, 7) && $-3827$ & 235/248 & 5 \\
(\ecl{709.a1}, 7) && $-3863$ & 244/255 &5 \\
 (\ecl{997.c1}, 7) && $-3811$ & 227/233& 6 \\ 
 (\ecl{1613.a1}, 7) &&$-4623$ & 282/290 & 5  \\
 (\ecl{1627.a1}, 7) &&$-4679$ & 294/298& 6\\
 (\ecl{6011.a1}, 7) &$\checkmark$& $-4591$ & 289/298 &  6 \\
 \hline
(\ecl{2251.a1}, 11) &$\checkmark$& $-3559$ & 205/235  & 5\\
\hline
 (\ecl{1933.a1}, 13) &$\checkmark$& $-4835$ &264/275 & 5 \\
 \hline
 (\ecl{709.a1}, 29)     &$\checkmark$& $-3012$& 204/218& 4\\
 \hline
(\ecl{1483.a1}, 31)    &$\checkmark$ & $-3080$ &203/224 & 4\\
\hline
\end{tabular}
\end{center}

\medskip 

The above data leads us to assume that $L_K$ is uniquely determined by the field $K$. Under this assumption for each $B\in \mathbb N$ we consider the set $\cD_B$ of all imaginary quadratic fields $K$ such that 
\begin{itemize}
\item the Heegner hypothesis holds for $E$ and $K$,
\item the analytic rank of the twisted curve $E^K/\Q$ is 1 and $p$ splits in $\cO_K$,
\item $\disc(K)\geq -B$, 
\end{itemize}

\noindent and define the function $ n_{(E, p)}: \mathbb N \to \mathbb N$ where
$$ n_{(E, p)}(B):= \min\{n\in \mathbb N \,| \, s_K \not\equiv s_{K'}  \pmod{p^n} \indent \forall  K, K'\in \cD_B\}.$$

\begin{question}
What can one say about the function $ n_{(E, p)}$?
\end{question}

\bigskip
\subsection*{Acknowledgements} 
We are grateful to Andrew Sutherland and Noam Elkies for helpful conversations.\\
\indent This material is based upon work supported by the National Science Foundation under grant agreements DMS-1440140,  DMS-1945452 (J.B.), DMS-1352598 (M.\c C.), DMS-1302409 and  DMS-1601028 (B.M). The first author was also supported by the Simons Foundation grant nos. 550023 and 1036361.\\
\indent The authors are grateful to MSRI/SLMath for their hospitality and support on multiple visits.



\end{document}

%% file: macros.tex
\usepackage{amsmath}
\usepackage{amsfonts}
\usepackage{amssymb}
\usepackage{amsthm}

\usepackage{url}

\usepackage{xspace}  

\DeclareMathOperator{\im}{im}

\DeclareMathOperator{\id}{id}

\newcommand{\bN}{\mathbf{N}}
\newcommand{\bF}{\mathbb{F}}

\DeclareFontEncoding{OT2}{}{} 

\newcommand{\cL}{\mathcal{L}}

\newcommand{\cO}{\mathcal{O}}

\DeclareMathOperator{\res}{res}

\DeclareMathOperator{\rk}{rk}

\newcommand{\into}{\rightarrow}

\DeclareMathOperator{\tors}{tors}

\newcommand{\comment}[1]{}
\newcommand{\Q}{\mathbb{Q}}

\newcommand{\cD}{\mathcal{D}}

\newcommand{\cS}{\mathcal{S}}

\newcommand{\K}{{\mathbb K}}

\newcommand{\cU}{\mathcal{U}}
\newcommand{\Z}{\mathbb{Z}}

\newcommand{\F}{\mathbb{F}}

\newcommand{\Zp}{\Z_p}

\renewcommand{\O}{\mathcal{O}}

\DeclareMathOperator{\disc}{disc}
\DeclareMathOperator{\ord}{ord}

\theoremstyle{plain}

\newtheorem{proposition}[equation]{Proposition}

\newtheorem{lemma}[equation]{Lemma}

\newtheorem{conjecture}[equation]{Conjecture}

\theoremstyle{definition}
\newtheorem{definition}[equation]{Definition}
\newtheorem{question}[equation]{Question}

\newtheorem*{Quest}{Question}
\newtheorem*{Quests}{Questions}

\newtheorem{remark}[equation]{Remark}

\newtheorem{example}[equation]{Example}

\numberwithin{equation}{section}
\numberwithin{figure}{section}
\numberwithin{table}{section}

\newcounter{listnum}

\theoremstyle{definition}
\newtheorem{algor}[equation]{Algorithm}
{\end{algor}}

{\begin{enumerate}\setlength{\itemsep}{0.1ex}}{\end{enumerate}}

\usepackage{color}
\usepackage{listings} 
\lstdefinelanguage{Sage}[]{Python}
{morekeywords={True,False,sage,singular},
sensitive=true}
\definecolor{lightyellow}{rgb}{1,1,.86}
\definecolor{dblackcolor}{rgb}{0.0,0.0,0.0}
\definecolor{dbluecolor}{rgb}{.01,.02,0.7}
\definecolor{dredcolor}{rgb}{0.8,0,0}
\definecolor{dgraycolor}{rgb}{0.30,0.3,0.30}
\definecolor{graycolor}{rgb}{0.35,0.35,0.35}
